\documentclass{article}
\usepackage[utf8]{inputenc}
\usepackage[T1]{fontenc}
\usepackage[margin=1.2in]{geometry}
\usepackage{lmodern}
\usepackage{subcaption}
\usepackage{mathtools}
\usepackage{amsthm}
\usepackage{todonotes}
\usepackage{mathrsfs,amssymb} 
\usepackage{tabularx}
\usepackage{multirow,array}
\usepackage{array}
\usepackage{makecell}
\usepackage{enumitem}

\newtheorem{theorem}{Theorem}
\newtheorem{definition}[theorem]{Definition}

\newtheorem{proposition}[theorem]{Proposition}
\newtheorem{lemma}[theorem]{Lemma}

\newtheorem{conjecture}[theorem]{Conjecture}

\newtheorem{problem}[theorem]{Problem}
\newtheorem{question}[theorem]{Question}
\theoremstyle{definition}
\newtheorem*{definition*}{Definition}

\usepackage{bookmark}
\newtheorem{innercustomthm}{Theorem}
\newenvironment{customthm}[1]
 {\renewcommand\theinnercustomthm{#1}\innercustomthm}
 {\endinnercustomthm}

\renewcommand{\Im}{Im}

\newcommand*{\sH}{\mathscr{H}}
\newcommand{\abs}[1]{\left|#1\right|}

\newcommand*{\myproofname}{Proof}

\title{Packing colourings in complete bipartite graphs and\\the inverse problem for correspondence packing}

\author{Stijn Cambie\thanks{Department of Computer Science, KU Leuven Campus Kulak-Kortrijk, 8500 Kortrijk, Belgium, supported by the Institute for Basic Science (IBS-R029-C4) and a postdoctoral fellowship by the Research Foundation Flanders (FWO) with grant number 1225224N.
E-mail: {\tt stijn.cambie@hotmail.com}},
Rimma H\"am\"al\"ainen\thanks{Discrete Mathematics Group, Institut f\"ur Mathematik, Technische Universit\"at Berlin, E-mail: {\tt haemael@math.tu-berlin.de}}
}

\begin{document}
\parindent=0cm

\maketitle

\begin{abstract}
Applications of graph colouring often involve taking restrictions into account, and it is desirable to have multiple (disjoint) solutions. In the optimal case, where there is a partition into disjoint colourings, we speak of a packing. However, even for complete bipartite graphs, the list chromatic number can be arbitrarily large, and its exact determination is generally difficult. For the packing variant, this question becomes even harder.

In this paper, we study the correspondence- and list packing numbers of (asymmetric) complete bipartite graphs. In the most asymmetric cases, Latin squares come into play. Our results show that every $z \in \mathbb Z^+ \setminus \{3\}$ can be equal to the correspondence packing number of a graph. We disprove a recent conjecture that relates the list packing number and the list flexibility number. Additionally, we improve the threshold functions for the correspondence packing variant.
\end{abstract}

\section{Introduction}
In a foundational paper introducing the concept of choosability, \cite{ERT80}, it was already noted that the list chromatic number of a bipartite (two-colourable) graph can be arbitrarily large, and in the symmetric case, $\chi_{\ell}(K_{n,n})\sim \log_2 n.$
Determining $\chi_{\ell}(K_{a,b})$ precisely for general $a,b>0$ is a difficult task, except for some very asymmetric cases.
In this paper, we study the same question for two related concepts. The formal definitions of these concepts are provided in Subsection\ref{subsec:not&def}, along with remarks on notation (and slight abuse of notation).

When attempting to optimise resource utilisation in resource allocation problems, or finding a satisfactory schedule for scheduling problems, one can describe the problem as a list- or correspondence-colouring problem. In such situations, it may be desirable to find multiple disjoint solutions. For example, employers may prefer a varied job schedule throughout the week, with a different schedule for each day. The natural question is to identify when the maximum number of disjoint solutions can be obtained, which is known as a 'packing' of colourings. 

As an example of a correspondence packing problem, consider a scenario where there are $n$ cohorts, and each cohort has $k$ students. 
For every student $S$ and every cohort $C$ that $S$ is not part of, we assume that there is at most one student $T \in C$ for which $S$ and $T$ decline to work together. The goal is to form $k$ disjoint groups of students, each group consisting of one student from every cohort, and with no conflicts within each group.
The conjecture of Yuster~\cite{Yus21} states that $k \ge n$ or $k \ge n+1$ depending on the parity of $n$, is a necessary and sufficient condition to ensure that this group partitioning is possible.

Exploring the suggestion of Alon, Fellows and Hare~\cite{AFH96}, Cambie et al.~\cite{CCDK21} defined the list packing number $\chi_{\ell}^\star$ and correspondence packing number $\chi_{c}^\star$ $25$ years later. 
For a fixed graph $G$, the numbers $\chi_{\ell}^\star(G)$
and $\chi_{c}^\star(G)$
represent the smallest integer $k$ such that for every assignment with lists of length $k$, one can find $k$ disjoint proper (list resp. correspondence) colourings.
With this terminology, we can state the conjecture of the example above as $\chi_{c}^\star(K_n)=n$ when $n$ is even and $\chi_{c}^\star(K_n)=n+1$ when $n$ is odd. In particular, the correspondence packing number would be even for every clique.

In~\cite[Thr.~3\&9]{CCDK21}, it was shown that $\chi_{\ell}^{\star}(G) \le \chi_{c}^{\star}(G) \le 2 \delta^{\star}(G)$, i.e. the packing numbers are bounded by twice the degeneracy.
Furthermore, it was observed in~\cite[Prop.~24]{CCDK21}
that for every $d$, there exist complete bipartite graphs with degeneracy $d$ for which $\chi_{c}^{\star}(G) =2d.$
In particular, the last result implies that every even positive integer can be expressed as $\chi_c^\star(G)$ for some graph $G$.
In future work~\cite{CCDK22+}, but also stated in the thesis of Cambie~\cite{cambie2022thesis}, by studying the cycle we know that $3$ cannot be expressed as $ \chi_c^\star(G)$ for some graph $G$, i.e. $3$ is not in the image $\Im(\chi_c^\star)$, the set of possible values of $\chi_c^{\star}(G)$ for any graph $G$.
Equivalently, once the correspondence packing number of a graph is larger than $2$, it is at least $4$.
A similar and even more surprising behaviour was studied in~\cite{KS15} where it was proven that the chromatic number of a non-bipartite quadrangulation of the $n$-dimensional real projective space is at least $n+2$. This was initially proven for $n=2$ by Youngs in \cite{Y96}.

Before our work, it was unknown whether any odd number (different from $1$) belongs to $\Im(\chi_c^\star)$ and as such the inverse problem for this problem felt appealing, giving a plausible better understanding on the concept of correspondence packing number.
In an inverse problem, one wonders for a graph parameter $p$ and a fixed integer $z$, if there exists a graph $G$ for which $p(G)=z$. One can solve all of these inverse problems at once by determining $\Im(p)$.
Some examples of inverse problems (restricted to the class of trees) are given in~\cite{CSW09}, but also related to graph colouring problems there are more examples, see e.g.~\cite{BH14, Y96}.
The main conclusion of our paper is that
\begin{theorem}\label{thr:main}
Except for $z=3$, for every $z \in \mathbb Z^+$ there exists a graph $G$ for which $\chi_c^\star(G)=z,$ i.e.
 $$\Im(\chi_c^\star)=\mathbb Z^+ \setminus \{3\}.$$
\end{theorem}

So while for most chromatic numbers, every positive integer can be attained by considering the complete graph, and it is conjectured that the correspondence packing number of complete graphs only attain even numbers (and $1$), when considering the correspondence packing number of graphs, every positive integer can be achieved, with the exception of $3$.

Returning to the result in~\cite[Prop.~24]{CCDK21}, we note that this proposition implies that there exists a threshold function $T(d)$ such that $\chi_{c}^{\star}(K_{d,t}) =2d$ if and only if $t \ge T(d).$
In~\cite[Prop.~24]{CCDK21} it was proven that $T(d) \le ((2d-1)!)^d.$
We will improve this result and determine $T(d)$ up to a polylogarithmic factor in Section~\ref{sec:thresholdfunction_2s}. Note that the determination of the analogue for the correspondence chromatic number~\cite{Mudrock18} also involves a similar polylogarithmic factor.
The main result of Section~\ref{sec:thresholdfunction_2s} is summarised in Theorem~\ref{thr:threshold_2d}. 
Here $N(d)$ denotes the number of Latin squares of order $d$ (see Subsection~\ref{subsec:not&def}) and $\log$ denotes the natural logarithm.

\begin{theorem}\label{thr:threshold_2d}
 Let $x(d)= \frac{ ((2d-1)!)^d }{\binom{2d-1}{d}^2((d-1)!)^d N(d)}$.
 Then 
 \begin{enumerate}
 \item \label{itm:1} $\chi_{c}^{\star}(K_{d,t}) <2d$ if $t<x(d)$ and 
 \item \label{itm:2} $\chi_{c}^{\star}(K_{d,t}) =2d$ if $ t \ge x(d)\cdot d \log\left( (2d-1)! \right ).$
 \end{enumerate} 
\end{theorem}

By monotonicity, there also exists a threshold function $T'(d)$ ensuring that $\chi_{c}^{\star}(K_{d,t}) \ge 2d-1$ if and only if $t \ge T'(d).$
In Section~\ref{sec:thresholdfunction_2s-1}, we use a more precise counting argument to obtain an upper bound for $T'(d)$.

\begin{theorem}\label{thr:threshold_2d-1}
For $$y(d)=\frac{((2d-2)!)^d}{N(d)\binom{2d-2}{d}\binom{2d-2}{d-1}\left( 2((d-1)!)^d-(d-1)^2\left( d-1+\frac1d\right)((d-2)!)^d\right)},$$

we have the following:

\begin{equation*}
\chi_{c}^{\star}(K_{d,t}) \ge 2d-1 ~~\text{if}~~ t>y(d) \cdot d \log\left( (2d-2)! \right ).
\end{equation*}

\end{theorem}

By observing that the upper bound for $T'(d)$ is smaller than the lower bound for $T(d)$, whenever $d \ge 3$, we conclude that every odd number which is at least $5$, equals the correspondence packing number of some complete bipartite graph.
Thus as a corollary, we conclude our main result, Theorem~\ref{thr:main}.
These computations and some more precise bounds are given in Section~\ref{sec:Im=Z3}.

Theorem~\ref{thr:threshold_2d} and Theorem~\ref{thr:threshold_2d-1} are both results on the correspondence packing number of asymmetric (complete) bipartite graphs. This may inspire the determination of list- and correspondence packing numbers of asymmetric bipartite graphs in more detail, as has been initiated for the list chromatic number~\cite{ACK20+, Zhu22+} and the correspondence chromatic number~\cite{CaKa20+, Mudrock18} before. This is an asymmetric case of some work in~\cite{ERT80}. 
For list packing, as a base case with one partition class being of size at most $3$, we prove in Section~\ref{sec:chi_ell^*(K3t)} that $\chi_{\ell}^\star(K_{3,t})=3+1_{t\ge 9}$ for every $t \ge 2.$
This refutes a recent conjecture~\cite[conj.~15]{KMMP22} on flexible list colouring, a concept introduced in~\cite{DNP19} where one aims to find a list-colouring that satisfies as many requests as possible.

Finally, some related problems and conjectures are given in Section~\ref{sec:conc}.

\subsection{Definitions, preliminaries and notation}\label{subsec:not&def}

Given a graph $G=(V,E)$, a \textit{list-assignment} $L$ of $G$ is a mapping $L \colon V \rightarrow 2^{\mathbb Z^+}$ where for every $v\in V(G)$ $L(v)$ is a subset of the positive integers (list of colours) associated to the vertex. Given $k\in\mathbb{Z}^{+}$, a \textit{k-list-assignment} is defined as a list-assignment in which all lists have cardinality $k$. A \textit{proper L-colouring} is defined as a mapping $c:V(G)\rightarrow \mathbb{Z}^{+}$ such that $c(v)\in L(v)$ for each $v\in V(G)$ and whenever $uv$ is an edge of $G$ we have $c(u)\neq c(v)$. In other words, $c$ defines a proper colouring of $V(G)$ such that each vertex is coloured by a colour of its list. The \textit{list chromatic number} $\chi_{\ell}(G)$ is defined as the least $k$ such that $G$ admits a proper $L$-colouring for any $k$-list-assignment $L$ of $G$. Note that $\chi_{\ell}(G)$ is always at least the chromatic number $\chi(G)$ for any graph $G$. 

Given a list-assignment $L$ of $G$, an \textit{$L$-packing} of size $k$ is defined as a collection of $k$ mutually disjoint $L$-colourings $\{c_1,...,c_k\}$ of $G$, that is, $c_i(v)\neq c_j(v)$ for any $i\neq j$ and any $v\in V(G)$. The $L$-packing is proper if each of the disjoint $L$-colourings is proper. The \textit{list (chromatic) packing number} $\chi_{\ell}^{\ast}(G)$ of $G$ is the least $k$ such that $G$ admits a proper $L$-packing of size $k$ for any $k$-list-assignment $L$ of $G$. We use the term \textit{$k$-fold} to denote a list-assignment with $\abs{L(v)}=k$ for every $v \in V(G)$. 
The list packing number $\chi_{\ell}^{\ast}(G)$ must be at least $\chi_{\ell}(G)$ for any graph $G$.

A \emph{cover} $H$ of G with respect to a mapping
$L \colon V(G) \to V(H)$ is a graph $H$ for which
$(L(v))_{v \in G}$ is a partition of $V(H)$ and the subgraph between
$L(u)$ and $L(v)$ is empty whenever $uv \not \in E(G).$ More formally, the cover is the pair $\sH=(L,H)$ and the vertex corresponding with $c \in L(v)$ is denoted $c_v.$
For formal definitions of these terms, see also~\cite{CCDK21, CCDK22+} (with the latter being more formal than the former).
To make the notation lighter, we refer to $L$ or $H$ in certain cases, when implicitly referring to $\sH=(L,H).$
E.g. when focusing on correspondence colouring, we assume that the list-assignment $L$ is such that when $\abs{L(v)}=k$ for $v \in V(G),$ then $L(v)=\{1, 2, \ldots, k\}=[k]$, and the associated vertices in $H$ will be presented as $\{1_v, 2_v, \ldots, k_v\}.$
Whenever necessary to know which vertex in $H$ we are referring to, we will also write $L(v)=\{1_v, 2_v, \ldots, k_v\}$ for clarity.

A \textit{correspondence-assignment} for a graph $G$ is a cover $H$ for $G$ via $L$ such that for each edge $uv\in E(G)$ the induced subgraph between $L(u)$ and $L(v)$ is a \textit{matching} (an independent edge set). In the latter case, we call $H$ a \textit{correspondence-cover} for $G$ with respect to $L$. A matching is a \emph{perfect matching} if it contains all the vertices in the graph. 
A \textit{correspondence L-colouring} of $H$ is an independent transversal (IT) of $H$ with respect to $L$, containing one vertex of $L(v)$ for every $v \in V(G).$ The \textit{correspondence chromatic number} (or \textit{DP chromatic number}) $\chi_{c}(G)$ is the least $k$ such that any $k$-fold correspondence-cover $H$ of $G$ via $L$ admits a correspondence-$L$-colouring, as defined in \cite{DvPo18}. Given a $k$-fold correspondence-cover, a \textit{$k$-fold correspondence L-packing} of $H$ is a $k$-fold IT-packing of $H$ with respect to $L$, i.e., $k$ disjoint correspondence-$L$-colourings.
The \textit{correspondence (chromatic) packing number} $\chi_{c}^{\ast}(G)$ of $G$ is defined as the least $k$ such that any $k$-fold correspondence-cover $H$ of $G$ via $L$ admits a $k$-fold correspondence-$L$-packing. Note that every list-cover can also be represented as a correspondence-cover. Furthermore, for any list-cover $H$ of $G$ via $L$, the correspondence-$L$-colourings of $H$ must be in bijective correspondence with the proper $L$-colourings of $G$. Note that $\chi_{c}(G)$ must be at least $\chi_{\ell}(G)$ and moreover, $\chi_{c}^{\ast}(G)$ must be at least $\max\{\chi_{\ell}^{\ast}(G),\chi_{c}(G) \}.$

\underline{Note on spelling.} Similar to~\cite{CCDK21,CCDK22+}, we write list colouring and correspondence colouring when referring to the topic and list-colouring and correspondence-colouring when referring to an instance. The same holds for the packing variants and the term list-assignment. The various packing numbers never get a hyphen.

We use the standard notation $[k]$ for the set $\{1,2,\ldots, k\}$ and $\binom{[k]}{m}$ for the family of all $m$-subsets of $[k].$
A permutation $\sigma$ of the set $[k]$ is called a derangement if for every $1 \le i \le k$, $\sigma(i) \neq i$. Two permutations $\sigma$ and $\tau$ of $[k]$ are said to be derangements of each other if $\sigma(i) \neq \tau(i)$ for every $1 \le i \le k$.

A Latin square of order $n$ is a $n \times n$ array $A = (a_{i,j})_{1 \le i,j \le n}$ with entries from
$[n]$ such that the entries in each row and in each column are distinct.
For any $k \le n$, when $A$ is a Latin square of order $n$, the array $A_{[k]\times [n]}$ is a $k \times n$ Latin rectangle.
Let $N(n)$ be the number of Latin squares of order $n$.
Already in $1975$, Alter~\cite{Alter75} asked to determine $N(n).$
Shao and Wei~\cite{SW92} found a formula in terms of binomials involving the permanent of binary $n \times n$ matrices and McKay and Wanless~\cite{MW05} had an alternative in terms of $n \times n$ $\{-1,+1\}$-matrices and computed it for $n$ up to $11.$
A trivial observation is that the number of Latin squares of order $n$ equals the number of $(n-1)\times n$ Latin rectangles, as there is a bijection between them (the final row is determined uniquely).

Finally, we recall Hall's marriage theorem:

\begin{theorem}\label{hall}(Hall \cite{Hal35})
Given a family $\mathcal{F}$ of finite subsets of some ground set $X$, where the subsets are counted with multiplicity, suppose $\mathcal{F}$ satisfies the \textit{marriage condition} i.e. for each $\mathcal{F}'\subseteq \mathcal{F}$ it holds that

\begin{equation*}
 |\mathcal{F}'|\le \Bigg|\bigcup_{A\in \mathcal{F}'} A\Bigg|.
\end{equation*}

Subsequently, there exists an injective function $f:\mathcal{F}\rightarrow X$ such that $f(A)$ is an element of the set $A$ for each $A\in\mathcal{F}$ i.e. the image $f(\mathcal{F})$ is a system of distinct representatives of $\mathcal{F}$.
\end{theorem}

\section{On the threshold function for $t$ ensuring that $\chi_c^\star(K_{d,t})=2d$}\label{sec:thresholdfunction_2s}

In this section, we will prove Theorem~\ref{thr:threshold_2d}.
Let $K_{d,t}=(U \cup V, E)$ with $U=\{u_1,u_2, \ldots ,u_d\}$ and $V=\{v_1,v_2,\ldots, v_t\}$ be a complete bipartite graph.
If we consider an $\ell$-fold correspondence-cover $H$ of $K_{d,t}$, then we assign lists $L(u_i)=\{1_{u_i}, 2_{u_i},\ldots, \ell_{u_i}\}$ and $L(v_j)=\{1_{v_j}, 2_{v_j},\ldots, \ell_{v_j}\}$ for every $i \in [d], j \in [t]$ and there are some matchings between the lists. 
As mentioned in Subsection~\ref{subsec:not&def}, we will omit subscripts in the lists of vertices in most of the remaining of this section.
This will allow us to work purely with permutations of $[\ell]$. 

Note that in the cover graph, one can add edges to every non-perfect matching between two lists and as such, extending them to perfect matchings. Since additional edges (constraints) imply that finding a correspondence-packing is harder, we can assume this is always the case.
Every matching between $L(u_i)$ and $L(v_j)$ corresponds with a permutation $\sigma_{i,j}$ on $[\ell]$ for which
$\{k, \sigma_{i,j}(k)\}=\{k_{u_i}, \sigma_{i,j}(k)_{v_j}\} \in E(H)$ for every $k \in [\ell].$

Let $c_1, \ldots, c_{\ell}$ be the $\ell$ colourings of $K_{d,t}$ that we aim to create in order to form a corresponding packing, meaning that they must be disjoint and proper.
For every vertex $u_i,$ we denote with the vector $\vec c(u_i)=( c_1(u_i), \ldots, c_{\ell}(u_i) )$ the different colours (vertices of the cover graph) assigned to $u_i$ by these $\ell$ colourings.
Omitting the subscript $u_i$, every $\vec c(u_i)$ will be a permutation of $[\ell].$
The same is true for the $v_j.$

A correspondence-packing will exist if and only if there is a choice for $\left(\vec c(u_i)\right)_{1 \le i \le d}$ such that for every $v_j,$ we can find a vector $\vec c(v_j)$ such that no neighbours in $H$ belong to the same colouring, i.e. for every $i \in [d], j \in [t],$ $\sigma_{i,j}\vec c(u_i)$ and $\vec c(v_j)$ are derangements.

If the matchings $\sigma_{i,j}$ (corresponding to a permutation of $[\ell]$)
 for $1 \le i \le d, 1 \le j \le t$ are the identity on $[\ell]$,
we call this the \emph{standard correspondence-cover} of $K_{d,t}$ (with list-sizes $\ell$).

\subsection{Prelimary explanation}\label{subsec:prelim_d2}

Since $\chi_c^\star(K_{d,t})\le 2d$ by~\cite[Prop.~24]{CCDK21}, we need to show that $\chi_c^\star(K_{d,t})> 2d-1$ for a particular $t$.

That is, we need to prove that there exists a $(2d-1)$-fold correspondence-cover for which no proper corresponding packing exists.
We assume that there are perfect matchings between $L(u_i)$ and $L(v_i)$ for every choice $(i,j) \in [d] \times [t]$ and that for every $k \in [2d-1]$ and $i\in [d]$ we have $\{k_{u_i}, k_{v_1} \}\in E(H).$

We now first point out what is happening in the case $d=2$ (for $d=1$ there is nothing to do).

\begin{figure}[h]
 \begin{center}
 \begin{tikzpicture}
 {
 
 \draw (0,1) ellipse (0.6cm and 1.35cm);
 \draw (0,5) ellipse (0.6cm and 1.35cm);
 \draw (5,1) ellipse (0.6cm and 1.35cm);
 \draw (5,5) ellipse (0.6cm and 1.35cm);
 
 \draw (-1,1) node {\large $v_1$};
 \draw (-1,5) node {\large $v_2$};
 \draw (6,1) node {\large $u_1$};
 \draw (6,5) node {\large $u_2$};
					
	\foreach \x in {0,1,2}{
	\draw[thick] (0,\x) -- (5,\x);
	\draw[thick] (0,\x) -- (5,4+\x);}
	
	\foreach \x in {0,1,2}{
	\draw[thick, style=dotted] (0,4+\x) -- (5,\x);
	}

	\draw[thick, style=dotted] (0,4)--(5,4);
\draw[thick, style=dotted] (0,5)--(5,6);
\draw[thick, style=dotted] (0,6)--(5,5);
	
		\foreach \x in {0,1,2,6,4,5}{\draw[fill] (0,\x) circle (0.1);
		\draw[fill] (5,\x) circle (0.1);}
				
	}
	\end{tikzpicture}
	\end{center}
 \caption{A correspondence-cover of $K_{2,2}=(\{u_1,u_2\} \cup \{v_1,v_2\} ,E)$ indicating that $\chi_c^\star(K_{2,2})>3$.}
 \label{fig:chicK22=4}
\end{figure}

Consider the partial correspondence-packing of $U$, i.e. $\vec c(u_i)=(c_1(u_i),c_2(u_i),c_3(u_i))$ is a permutation of $L(u_i)$ for every $u_i \in U.$
This partial packing can be extended at $v_1$ if we can find a permutation $\vec c(v_1)$ that is a derangement of both $\vec c(u_1)$ and $\vec c(u_2).$
Let $S_o=\{(2,1,3), (1,3,2),(3,2,1)\}$ be the set of odd permutations and $S_e=\{(1,2,3), (2,3,1),(3,1,2)\}$ be the set of even permutations of $(1,2,3).$
By Hall's matching theorem~\ref{hall}, such an extension is not possible if and only if there are coordinates $i,j$ such that $c_i(u_1)=c_j(u_2)$ and $c_j(u_1)=c_i(u_2)$.
This is equivalent with
$\vec c(u_1)\vec c(u_2)$ being an odd permutation of $[3]$, i.e. $\vec c(u_1) \in S_o$ and $\vec c(u_2) \in S_e$ or $\vec c(u_1) \in S_e$ and $\vec c(u_2) \in S_o.$ 
There are $18$ choices for $\vec c(u_1), \vec c(u_2)$ satisfying this, so $18$ out of $(3!)^2=36$ possibilities in total for the pair of permutations $\vec c(u_1), \vec c(u_2)$ cannot be extended in $v_1$ to form a partial packing.
The same counting argument is true for $v_2$.
Since the $2$ sets of bad pairs of permutations can be disjoint, e.g. when the matchings are as given in Figure~\ref{fig:chicK22=4}, we conclude that $T(2)=2.$
In the latter case, the partial packing cannot be extended in $v_2$ if and only if $\vec c(u_1), \vec c(u_2) \in S_o$ or $\vec c(u_1), \vec c(u_2) \in S_e$.
For full understanding of this base case, in Figure~\ref{fig:table36permutations} we have listed all possibilities of $\vec c(u_1)$ and $\vec c(u_2)$, with the two vectors presented in the columns of a matrix, classified according to which one of the two vertices in $V$ cannot be extended.

\begin{figure}[h]
\begin{center}
 \begin{tabular}{ | l | c | r | }
 \hline
 no choice for $\vec c(v_1)$ & \makecell{\vspace{0.1\baselineskip}\\$\begin{pmatrix}
 1 & 2 \\
 2 & 1 \\
 3 & 3
 \end{pmatrix}$,$\begin{pmatrix}
 1 & 1 \\
 2 & 3 \\
 3 & 2
 \end{pmatrix}$,$\begin{pmatrix}
 1 & 3\\
 2 & 2\\
 3 & 1
 \end{pmatrix}$,$\begin{pmatrix}
 2 & 2 \\
 3 & 1\\
 1 & 3
 \end{pmatrix}$,$\begin{pmatrix}
 2 & 1 \\
 3 & 3\\
 1 & 2
 \end{pmatrix}$, $\begin{pmatrix}
 2 & 3 \\
 3 & 2 \\
 1 & 1
 \end{pmatrix}$
 \\\vspace{0.1\baselineskip}\\
 $\begin{pmatrix}
 3 & 2 \\
 1 & 1 \\
 2 & 3
 \end{pmatrix}$,$\begin{pmatrix}
 3 & 1\\
 1 & 3\\
 2 & 2
 \end{pmatrix}$,$\begin{pmatrix}
 3 & 3\\
 1 & 2\\
 2 & 1
 \end{pmatrix}$,$\begin{pmatrix}
 2 & 1 \\
 1 & 2\\
 3 & 3
 \end{pmatrix}$,$\begin{pmatrix}
 2 & 2\\
 1 & 3\\
 3 & 1
 \end{pmatrix}$,$\begin{pmatrix}
 2 & 3\\
 1 & 1\\
 3 & 2
 \end{pmatrix}$\\\vspace{0.1\baselineskip}\\$\begin{pmatrix}
 1 & 1 \\
 3 & 2\\
 2 & 3
 \end{pmatrix}$,$\begin{pmatrix}
 1 & 2 \\
 3 & 3 \\
 2 & 1
 \end{pmatrix}$, $\begin{pmatrix}
 1 & 3 \\
 3 & 1 \\
 2 & 2
 \end{pmatrix}$,$\begin{pmatrix}
 3 & 1\\
 2 & 2\\
 1 & 3
 \end{pmatrix}$,$\begin{pmatrix}
 3 & 2\\
 2 & 3\\
 1 & 1
 \end{pmatrix}$,$\begin{pmatrix}
 3 & 3\\
 2 & 1\\
 1 & 2
 \end{pmatrix}$\\\vspace{0.1\baselineskip}}\\
 \hline
 no choice for $\vec c(v_2)$ & \makecell{\vspace{0.1\baselineskip}\\$\begin{pmatrix}
 1 & 1 \\
 2 & 2 \\
 3 & 3
 \end{pmatrix}$,$\begin{pmatrix}
 1 & 2 \\
 2 & 3 \\
 3 & 1
 \end{pmatrix}$,$\begin{pmatrix}
 1 & 3\\
 2 & 1\\
 3 & 2
 \end{pmatrix}$,$\begin{pmatrix}
 2 & 1 \\
 3 & 2\\
 1 & 3
 \end{pmatrix}$,$\begin{pmatrix}
 2 & 2 \\
 3 & 3\\
 1 & 1
 \end{pmatrix}$, $\begin{pmatrix}
 2 & 3 \\
 3 & 1 \\
 1 & 2
 \end{pmatrix}$
 \\\vspace{0.1\baselineskip}\\
 $\begin{pmatrix}
 3 & 1 \\
 1 & 2 \\
 2 & 3
 \end{pmatrix}$,$\begin{pmatrix}
 3 & 2\\
 1 & 3\\
 2 & 1
 \end{pmatrix}$,$\begin{pmatrix}
 3 & 3\\
 1 & 1\\
 2 & 2
 \end{pmatrix}$,$\begin{pmatrix}
 2 & 2 \\
 1 & 1\\
 3 & 3
 \end{pmatrix}$,$\begin{pmatrix}
 2 & 1\\
 1 & 3\\
 3 & 2
 \end{pmatrix}$,$\begin{pmatrix}
 2 & 3\\
 1 & 2\\
 3 & 1
 \end{pmatrix}$\\\vspace{0.1\baselineskip}\\$\begin{pmatrix}
 1 & 2 \\
 3 & 1\\
 2 & 3
 \end{pmatrix}$,$\begin{pmatrix}
 1 & 1 \\
 3 & 3 \\
 2 & 2
 \end{pmatrix}$, $\begin{pmatrix}
 1 & 3 \\
 3 & 2 \\
 2 & 1
 \end{pmatrix}$,$\begin{pmatrix}
 3 & 2\\
 2 & 1\\
 1 & 3
 \end{pmatrix}$,$\begin{pmatrix}
 3 & 1\\
 2 & 3\\
 1 & 2
 \end{pmatrix}$,$\begin{pmatrix}
 3 & 3\\
 2 & 2\\
 1 & 1
 \end{pmatrix}$\\\vspace{0.1\baselineskip}}\\
 \hline
 \end{tabular}
\end{center}
\caption{A table of choices for $\vec c(u_1)~\text{and}~\vec c(u_2)$ for which the partial packing does not extend in $v_1$ and respectively in $v_2$.}\label{fig:table36permutations}
\end{figure}

In the next subsection, we apply the same idea for general $d$.
Nevertheless, for $d \ge 3$ it is much harder to construct the best set of matchings between every $L(v_j)$ and the $L(u_i)$ and as such we only give an estimate of the threshold function. 
As evidence for that, we note that $x(d)$ is not always integral for large $d$ (e.g. for $7 \le d \le 11$).

\subsection{Approximation of the threshold function $T(d)$ for $d \ge 3$}

As was the case in the previous subsection, we focus on the forbidden partial packings of the bipartition class of size $d$, given the matchings with one particular vertex $v$. For the reader's convenience, we will repeat the statement of theorem~\ref{thr:threshold_2d}.

\begin{customthm} {\bf \ref{thr:threshold_2d}}
 Let $x(d)= \frac{ ((2d-1)!)^d }{\binom{2d-1}{d}^2((d-1)!)^d N(d)}$.
 Then 
 \begin{enumerate}
 \item $\chi_{c}^{\star}(K_{d,t}) <2d$ if $t<x(d)$ and 
 \item $\chi_{c}^{\star}(K_{d,t}) =2d$ if $ t \ge x(d)\cdot d \log\left( (2d-1)! \right ).$
 \end{enumerate} 
\end{customthm}

\begin{proof}[Proof of Theorem~\ref{thr:threshold_2d}~\eqref{itm:1}]

Take any $t < x(d)$ and correspondence-cover of $K_{d,t}=(U \cup V,E).$ Let $U=\{u_1,u_2, \ldots ,u_d\}$ and $V=\{v_1,v_2, \ldots ,v_t\}.$
If some matching is not a perfect matching, we can extend it to a perfect matching. If there was no correspondence-packing for a cover with partial matchings, then there will not be none in the one with extended matchings.

We now focus on one vertex $v := v_j$.
Since the counting argument is not influenced by the order of the elements in the lists when defining the matchings, we can assume that the edges in the cover of $K_{1,d}=(U \cup \{v\},E')$ are precisely $\{j_{u_i},j_v\}$ for $1 \le i \le d, 1 \le j \le 2d-1,$ i.e., the standard perfect matching.
By Hall's Marriage theorem, the partial correspondence-packing of $U$ cannot be extended in $v$ if and only if there is a subset of $d$ integers which can only be assigned to $d-1$ indices (i.e. colourings) of $\vec c(v).$
This is immediate from the graph version as stated in~\cite[Lem.~17]{CCDK22+}.
If we consider the $(2d-1)\times d$-matrices whose columns correspond with the vectors $(\vec c(u_i))_{i \in [d]}$, the latter condition is equivalent with the existence of subsets $C,J \subset [2d-1]$ of size $d$ for which 
$(c_j(u_i))_{i \in [d], j \in J}$ form a Latin square of order $d$ whose values are in $C$.
Note that no partial correspondence-packing can lead to a contradiction for two different choices of $(C,J).$
There are the total of $\binom{2d-1}{d}^2$ choices for $(C,J)$ and
$N(d)$ choices for the Latin square $(c_j(u_i))_{i \in [d], j \in J}$ once $(C,J)$ are fixed as well as $((d-1)!)^d$ choices for $(c_j(u_i))_{i \in [d], j \in [2d-1]\setminus J}$.
This leads to $$w(d)=\binom{2d-1}{d}^2((d-1)!)^d N(d)$$ partial correspondence-packings that cannot be extended in $v.$
Since this argument was for any choice of $1 \le j \le t,$ the same counting applies to every $v_j.$

We know that the number of partial correspondence-packings that cannot be extended in some $v_j$ is bounded by $t \cdot w(d)< ((2d-1)!)^d, $
as such there is a choice for the partial correspondence-packing on $U$ that can be extended on $V$, i.e. there exists a correspondence-packing.
This implies that $\chi_c^\star(K_{d,t}) \le 2d-1.$
\end{proof}

\begin{proof}[Proof of Theorem~\ref{thr:threshold_2d}~\eqref{itm:2}]
Consider all choices of $d$ possible perfect matchings between a list $L(v)$ and the $d$ lists $L(u_i),$ $1\le i \le d.$
By the proof of Theorem~\ref{thr:threshold_2d}~\eqref{itm:1}, we know that for each combination there are $w(d)$ partial correspondence-packings that cannot be extended in $v.$
Furthermore, as we are ranging over all possible combinations of perfect matchings, every partial correspondence-packing of $U$ cannot be extended the same number of times.

Now, we construct an iterative procedure for a correspondence-cover of $K_{d,t}$ with list-sizes $2d-1$ for some $t \le x(d)\cdot d \log\left( (2d-1)! \right ) $ for which no correspondence-packing exists.
Iteratively, we choose a combination of perfect matchings between $L(v_s)$ and each of the $L(u_i)$.
Let $X_s$ be the number of partial correspondence-packings of $U$ that are extendable in $\{v_1,\ldots, v_{s}\}$.
We will write $w=w(d)$ and $x=x(d)$ in the remaining of the proof.
We will prove the inequality $X_s \le (1-1/x)^s \cdot X_0$ by induction, where $X_0=((2d-1)!)^d.$
There are $w=X_0/x$ partial correspondence-packings forbidden by $u_1$ and thus $X_1=(1-1/x)X_0$. 

On average, $1/x$ of the partial correspondence-packings are forbidden by a combination of perfect matchings. As such, there is a choice of perfect matchings between $L(v_s)$ and each of the $L(u_i)$, for which at least this portion of the remaining $X_{s-1}$ partial correspondence-packings are forbidden.
Choosing this combination of perfect matchings, we conclude that $X_s \le (1-1/x)X_{s-1}\le (1-1/x)^s X_0$ where the second inequality follows from the induction hypothesis.
As such our induction step is complete and hence the result.
Repeating this $\log_{ 1-1/x } \left( x/X_0 \right)$ times, there are at most $x$ partial correspondence-packings which can be extended.
Since the number of forbidden partial correspondence-packings in every step is an integer, we see that with no more than 
\begin{align*}
 \log_{ 1-\frac{1}{x} } \left( \frac{x}{X_0} \right) + x &= \frac{ \log(x)-\log(X_0)}{\log \left(1-\frac{1}{x} \right)} +x\\
 &\le x \left( \log(X_0)-\log(x) +1 \right)\\
 &< x \log(X_0)
\end{align*} 
vertices in $V$, we can construct a correspondence-cover of $K_{d,t}$ which does not admit a correspondence-packing.
Here we used that $\log{ \left(1-1/x \right)} <-1/x $ and $x>3$ when $d \ge 3$ (and thus $\log x >1$).
\end{proof}

\section{On the threshold function for $t$ ensuring that $\chi_c^\star(K_{d,t})\ge 2d-1$}\label{sec:thresholdfunction_2s-1}

In this section, we will establish an upper bound on the threshold value of $t$ such that $\chi_c^\star(K_{d,t})\geq 2d-1$. To do so, we will examine matrices $(\vec{c}(u_i))_{i\in[d]}$, and determine the number of matrices for which the partial packing cannot be extended to $v_1$. We will refer to such matrices as \emph{forbidden}.

\begin{definition}
 We call a $d \times k$-matrix all of whose columns are permutations of $[k]$ a \emph{$d \times k$-packing-matrix}.
 A $d \times k$-packing-matrix $M$ is \emph{forbidden} if there does not exist a permutation of $[k]$ which is a derangement of every column of $M$.
\end{definition}

We will use the following specific version of Hall's matching theorem, as stated in~\cite[Lem.~18]{CCDK22+}, when counting the number of forbidden $d \times (2d-2)$-packing matrices.

\begin{lemma}\label{lem:checkingHallCondition_graphversion_2}
	Let $G=(A \cup B,E)$ be a bipartite graph with partite sets $A$ and $B$ such that $\lvert A \rvert = \lvert B \rvert = 2m$ with minimum degree $m-1$ for some $m \ge 2$.
	Then for every $A_1 \subseteq A$, we have that $\lvert N(A_1) \rvert \ge \lvert A_1 \rvert$ except possibly for 
	\begin{enumerate}
		\item $\lvert A_1 \rvert =m$ and $\lvert N(A_1) \rvert=m-1$
		\item $\lvert A_1 \rvert =m+1$ and $\lvert N(A_1) \rvert=m-1$
		\item $\lvert A_1 \rvert =m+1$ and $\lvert N(A_1) \rvert=m$
	\end{enumerate}
	Let $A_2=A \setminus A_1$, $B_1=N(A_1)$, and $B_2=B\setminus B_1$. 
	Then $G[A_1,B_2]$ is the empty graph, and 
	\begin{enumerate}
		\item $G[A_1,B_1]\cong K_{m,m-1}$,
		\item $G[A_1,B_1]\cong K_{m+1,m-1}$ and $G[A_2,B_2]\cong K_{m-1,m+1}$,
		\item $G[A_2,B_2]\cong K_{m-1,m}$
	\end{enumerate}
\end{lemma}

\begin{lemma}\label{lem:forbid_d_2d-2_pm}
 Fix $d \ge 3.$ 
 Then the number of forbidden $d \times (2d-2)$-packing matrices equals to
 $$N(d)\binom{2d-2}{d}\binom{2d-2}{d-1}\left( 2((d-1)!)^d-(d-1)^2\left( d-1+\frac1d\right)((d-2)!)^d\right).$$
\end{lemma}

\begin{proof}
By Lemma~\ref{lem:checkingHallCondition_graphversion_2}, there are three types of forbidden $d \times (2d-2)$-packing matrices. 
The three types correspond with $(d-1,d-2), (d,d-2)$ and $(d,d-1)$ obstructions, where the $(a,b)$ obstruction corresponds with $a$ entries in the final vector (which needs to be a permutation of $[2d-2]$) that only can take on (exactly) $b$ values from $[2d-2]$ in a common derangement of every column of the matrix.
 Moreover, each forbidden $d \times (2d-2)$-packing matrix has only one maximal obstruction.
 Note e.g. that if such a matrix has a $(d,d-2)$ obstruction, then it also has a (multiple) $(d-1,d-2)$ obstruction(s).
 Examples of $4 \times 6$-packing matrices with respectively a maximal $(3,2),(4,2)$ and $(4,3)$ obstruction are presented in Figure~\ref{fig:3casesforbiddenpackingmatrices} (the obstructing rows are the top ones).
 
 We now count the number of forbidden $d \times (2d-2)$-packing matrices $M$.
 A $(d-1,d-2)$ obstruction occurs precisely when there are $d-1$ rows containing the same $d$ numbers.
 Let $C, J \subset [2d-2]$ with $\lvert C \rvert =d$ and $\lvert J \rvert =d-1$ be the sets that represent the $d$ numbers and the $d-1$ row indices respectively.
 The restriction of the matrix $M_{J \times [d]}$ is a $(d-1) \times d$ Latin rectangle all whose entries belong to $C$.
 There are $\binom{2d-2}{d}\binom{2d-2}{d-1}$ choices to choose the sets $C$ and $J$. Having chosen these, there are $N(d)$ choices for $M_{J \times [d]}$ and then $((d-1)!)^d$ possibilities for $M$ given $M_{J \times [d]}$.
 By the product formula, we counted 
 $$w_1=N(d)\binom{2d-2}{d}\binom{2d-2}{d-1}((d-1)!)^d$$ of them. Nevertheless, every matrix with a $(d,d-2)$ obstruction has been counted $d$ times.
 
 Next, analogously, we count the number of forbidden $d \times (2d-2)$-packing matrices $M$ with a $(d,d-2)$ obstruction. In that case, there are sets $C, J \subset [2d-2]$ with $\lvert C \rvert =d=\lvert J \rvert$ such that $M_{J \times [d]}$ is a $d \times d$ Latin square all whose entries belong to $C$.
 There are $\binom{2d-2}{d}^2$ choices for $C,J$.
 For every choice there are $N(d)$ possibilities for $M_{J \times [d]}$ and $((d-2)!)^d$ possibilities to complete $M.$
 By the product formula, we have counted 
 $$w_2=N(d)\binom{2d-2}{d}^2((d-2)!)^d$$ of them.

 Finally, we count the number of forbidden $d \times (2d-2)$-packing matrices $M$ with a $(d,d-1)$ obstruction.
 Now there are sets $C, J \subset [2d-2]$ with $\lvert C \rvert =d-1$ and $d=\lvert J \rvert$ such that every row in $M_{J \times [d]}$ contains every number of $C.$
 As such, every row contains one additional number which does not belong to $C$.
 There are $N(d)$ choices for $M_{J \times [d]}$ given $C$ and $J$ when the additional numbers are discarded (or considered as a new element).
 If $d-1$ or $d$ additional numbers in these $d$ rows are the same, the matrix $M$ would have a $(d-1,d-2)$ obstruction.
 There are $d-1$ choices for the additional number, if it is the same in all rows.
 There are $(d-1)(d-2)$ choices for the two additional numbers, with $d$ possible entries where the number that appears only once can be.
 This gives the factor $d-1+(d-1)(d-2)d=(d-1)^3.$
 Having said all of this, we can compute the number $w_3$ of forbidden $d \times (2d-2)$-packing matrices $M$ with a $(d,d-1)$ obstruction without a $(d-1,d-2)$ obstruction;
 \begin{align*}
 w_3&=N(d)\binom{2d-2}{d}\binom{2d-2}{d-1}\left(
 ((d-1)!)^d-(d-1)^3((d-2)!)^d \right).
 \end{align*}
 Using the binomial identity $\binom{2d-2}{d}=\frac{d-1}{d}\binom{2d-2}{d-1}$,
 the total number of forbidden $d \times (2d-2)$-packing matrices equals the following:
 $$w_1-(d-1)w_2+w_3=
 N(d)\binom{2d-2}{d}\binom{2d-2}{d-1}\left( 2((d-1)!)^d-(d-1)^2\left( d-1+\frac1d\right)((d-2)!)^d\right).$$
 Note that in the precise counting, $w_1$ does count some of the $d \times (2d-2)$-packing matrices with a maximal $(d,d-1)$ obstruction.
\end{proof}

\begin{figure}[h]
 \centering
 $\begin{pmatrix} 
1 & 2 & 3 & 4\\
2 & 1&4&3 \\
3 & 4 & 2 & 1\\
4&3&5&6 \\
5&6&1 & 5 \\
6&5&6&2
\end{pmatrix}$
\quad
$\begin{pmatrix} 
1 & 2 & 3 & 4\\
2 & 1&4&3 \\
3 & 4 & 2 & 1\\
4&3&1&2 \\
5&6&5 & 5 \\
6&5&6&6
\end{pmatrix}$
\quad
$\begin{pmatrix} 
1 & 2 & 3 & 4\\
2 & 1&4& 3 \\
3 & 6 & 2 & 1\\
6&3&1&2 \\
5&4&5 & 5 \\
4&5&6&6
\end{pmatrix}$
\caption{Examples of the three types of forbidden $4\times 6$-packing-matrices.}
 \label{fig:3casesforbiddenpackingmatrices}
\end{figure}

Having proven Lemma~\ref{lem:forbid_d_2d-2_pm}, the proof of Theorem~\ref{thr:threshold_2d-1} is similar to the proof of Theorem~\ref{thr:threshold_2d}.

\begin{proof}[Proof of Theorem~\ref{thr:threshold_2d-1}]
Let

\begin{equation*}
w:=w(d)=N(d)\binom{2d-2}{d}\binom{2d-2}{d-1}\left( 2((d-1)!)^d-(d-1)^2\left( d-1+\frac1d\right)((d-2)!)^d\right)
\end{equation*}

and

\begin{equation*}
y:=y(d)=\frac{((2d-2)!)^d}{w(d)}.
\end{equation*}

We iteratively construct a correspondence-cover of $K_{d,t}=(U \cup V,E)$ for some $t \le y \cdot d \log\left( (2d-2)! \right )$ for which no correspondence-packing exists. Let $U=\{u_1,u_2, \ldots ,u_d\}$ and $V=\{v_1,v_2, \ldots ,v_t\}.$
Let $Y_s$ be the number of partial correspondence-packings of $U$ that are extendable in $\{v_1,\ldots, v_{s}\}$, which implies $Y_0=((2d-2)!)^d$.

If for every choice of $d$ perfect matchings between a list $L(v)$ and the $(L(u_i))_{1\le i \le d}$, we list the partial correspondence-packings of $U$ that cannot be extended according to $L(v)$, then each possibility will be listed the same number of times.
On average, $1/y$ of the partial correspondence-packings are forbidden, which implies that we can forbid all of them with no more than $\log_{1-1/y}\left(y/Y_0\right)+y \le y \log{Y_0}$ vertices in $V$ (again, $y(d) >3$ for $d\ge3$).
\end{proof}

\section{$\Im( \chi_c^\star)= \mathbb Z^+ \setminus\{3\}$}\label{sec:Im=Z3}

In Sections~\ref{sec:thresholdfunction_2s} and~\ref{sec:thresholdfunction_2s-1}, we derived bounds for $T(d)$ and an upper bound for $T'(d)$.
In Figure~\ref{fig:tableTvalues} we give a table with some computed bounds for $T'(d)$ and $T(d)$ whenever $3 \le d \le 11.$
In particular we note that for each $3 \le d \le 11$, there are complete bipartite graphs for which $\chi_c^\star(K_{d,t})= 2d-1$, since $T'(d)<T(d)$
For $d\in \{3,4\}$, we could slightly improve the bound by explicitly taking $X_s= \lfloor (1-1/x )X_{s-1} \rfloor$ instead of the estimate (which is given between brackets).

\begin{figure}[h]
\begin{center}
 \begin{tabular}{ | l | l | l | }
 \hline
 $d$& Upper bound $T'(d)$ & Lower bound $T(d)$\\
 \hline
 3& 54 (62) & 180\\ 
 4& 14853 (15172) & 705600\\
5&413809958& 308629440000\\
6&551649401930292& 
7808216194437120000\\
7&$5.97\cdot 10^{22}$ &$1.99\cdot 10^{28}$\\
8&$4.73\ \cdot 10^{32}$ &$4.55\cdot 10^{39}$\\
9& $3.02 \cdot 10^{44}$& $9.90 \cdot 10^{53}$\\
10&$1.63\cdot 10^{58}$ & $2.10\cdot 10^{68}$ \\
11&$7.72 \cdot 10^{73}$& $4.45\cdot10^{85}$\\
 \hline
 \end{tabular}
\end{center}
\caption{Estimates of threshold functions $T'(d)$ and $T(d)$ for which $\chi_c^\star(K_{d,t})\ge 2d-1$ resp. $2d$}\label{fig:tableTvalues}
\end{figure}

By finding explicit correspondence-covers (randomly generated), with a computer program, we can give more precise bounds for $T'(3)$ and $T(3).$

\begin{proposition}\label{prop:TT'3}
 We have $9 \le T'(3)\le 16$ and $180 \le T(3)\le 389.$
\end{proposition}

For $7 \le d \le 11$, the lower bound $x(d)$ for $T(d)$ was not integral since $N(d)$ contains a prime which is strictly larger than $2d-1.$
Therefore, at least for these cases, and possibly for $d\geq3$, there are no nice partitions as there were in the case $d=2$. 

For $d\ge 12$, the exact value of $N(d)$ is not known, but we can conclude that $T'(d)<T(d)$ for every $d \ge 3.$
The latter being done with elementary estimates and $2d^2 \log(2d)<(2d-1)^d$ for $d \ge 12.$

We conclude that for every $d \ge 3$, both $2d-1$ and $2d$ can be equal to $\chi_c^\star(G)$ for some (complete, bipartite) graph $G$, so do $1,2$ and $4$ (e.g. attained by $K_1, K_2$ and $K_3$).
Nevertheless, there does not exist a graph $G$ for which $\chi_c^\star(G)=3$ and so we conclude this is the only value that cannot be the correspondence packing number of a graph. Theorem~\ref{thr:main} is proven.

\section{The list packing number of $K_{3,t}$}\label{sec:chi_ell^*(K3t)}

A classical result says that $\chi_{\ell}(K_{d,t})=d+1$ if and only $t \ge d^d$.
This is easy to prove by considering a $d$-list-assignment of $K_{d,t}$ and considering the $d$ lists of the bipartition class $U$ of size $d$.
If they are not disjoint, we can colour the vertices in $U$ with at most $d-1$ colours and extend greedily (independent of the value of $t$) for the vertices of the other partition class, $V$.
In the other case, there are $d^d$ possible colourings on $U$ and every list of a vertex in $V$ forbids at most one of them.

Generalising this result to list packing is harder.
Complete bipartite graphs give an indication that list packing may conceptually be harder than list colouring.
For a fixed list-assignment $L$ of $K_{a,b}=(U \cup V,E)$, a list-colouring corresponds with a partition of the colours in two sets, each of them intersecting every list for every vertex of one vertex class. To find a list-packing, one needs to have permutations of every list such that $\vec{c}(u)$ and $\vec{c}(v)$ are derangements of each other for every $u \in U, v \in V.$

If the lists on the smaller partition class are disjoint, then it is still easy to find a list-packing. However, in the case where the lists are not disjoint, a list-packing may not exist. Additionally, simply considering the number of partial packings of $U$ that are forbidden by a single list at $V$ does not necessarily imply the result. The latter can be concluded from observing the matrix $A_1$ from Proposition~\ref{prop:chilK3t}.

In this section, we determine $\chi_{\ell}^\star(K_{a,t})$ for $a \in \{1,2,3\}$ and $ t \in \mathbb Z^+$.
For $a \le 2,$ this was already known: $\chi_{\ell}^\star(K_{1,t})=2$ for every $t$ and $\chi_{\ell}^\star(K_{2,t})=3$ for every $t \ge 2.$

For $a=3,$ we conclude by doing a careful case analysis.

\begin{proposition}\label{prop:chilK3t}
 The following list packing numbers are known
\begin{equation*}
 \chi_{\ell}^{\star}(K_{3,t})=\begin{cases}
 3 \mbox{ if } 2 \le t \le 8.\\
 4 \mbox{ if } t \ge 9.
\end{cases} 
\end{equation*}
\end{proposition}

\begin{proof}
Note that $3= \chi_{\ell}^\star(K_{2,2})\le \chi_{\ell}^\star(K_{3,t})\le 4$ by~\cite[Lem.~33]{CCDK21} for every $t\ge 3.$ Let $K_{3,t}=(U \cup V,E)$ where $\abs U=3.$
Due to monotonicity in $t$, it will be sufficient to prove that $\chi_{\ell}^{\star}(K_{3,8})=3$ and $\chi_{\ell}^{\star}(K_{3,9})=4.$
The second equality is easy.
Let $L(u_i)=\{ 3\cdot(i-1)+1, 3\cdot(i-1)+2, 3\cdot(i-1)+3\}$ for every $i \in [3]$ and $(L(v_j))_{1 \le j \le 9}=7 \times \{ 1,2,3\} \times \{4,5,6\}=\{ \{a,b,7\} \mid a \in [3], 4\le b \le 6\}.$
Consider a partial packing $\vec c$ of $U$, where without loss of generality, $\vec{c}(u_3)=(7,8,9)$. 
Now there is an index $j \in [9]$ for which $L(v_j)=\{ c_1(u_1), c_1(u_2), 7\}$ and we conclude that the partial packing cannot be extended in $v_j$ since colouring $c_1$ cannot be proper.

Next, we prove that $\chi_{\ell}^{\star}(K_{3,8})=3$.
Up to isomorphism, there are $12$ possible assignments of lists of length $3$ to the three vertices of $U.$
In our analysis, we write their lists as the rows of a $3 \times 3$-matrix. That is, we write $\vec{c}(u_i)$ as a row vector and consider the plausible $3 \times 3$-matrices $M=(\vec{c}(u_i))_{1 \le i \le 3}$. Given $(L(u_i))_{1 \le i \le 3}$, there are $(3!)^3=216$ possibilities for $M$. Since we can fix $\vec{c}(u_1)$, it is sufficient to restrict to the $36$ choices of $M$ for which the first row equals the fixed choice.
We call such a permuted matrix (elements in second and third row permuted) \emph{forbidden} by a list $L(v_j)$ if no permutation of $L(v_j)$ is a derangement of each row of the permuted matrix.
Now, we check the $12$ possible assignments on $U$ and present them by giving at least one plausible matrix $M$.

We first can consider the following $5$ matrices 

\begin{equation*}
A_1:= \begin{pmatrix}
 1 & 2&3 \\
 1 & 2&4 \\
 1 & 3&4
 \end{pmatrix}, A_2:= \begin{pmatrix}
 1 & 2&3 \\
 1 & 2&4 \\
 1 & 5&3
 \end{pmatrix},
 A_3:= \begin{pmatrix}
 3 & 2&1 \\
 2 & 4&1 \\
 3 & 4&5
 \end{pmatrix}, A_4:=\begin{pmatrix}
 1 & 2&3 \\
 1 & 2&4 \\
 6 & 5&3
 \end{pmatrix}, A_5:=\begin{pmatrix}
 1 & 2&4 \\
 1 & 5&3 \\
 6 & 2&3
 \end{pmatrix}.
 \end{equation*}

For each of these, we observe that no element appears in every column, no row contains three different elements, and no two numbers appear in the same two columns. 
Hence for every $v_i \in V$, we can apply Hall's condition to permute $L(v_i)$ in a common derangement of the three rows of the matrix.

Next, we consider the case 
 \begin{equation*} A_6:=
 \begin{pmatrix}
 1 & 2&3 \\
 1 & 2&4 \\
 1 & 2&5
 \end{pmatrix} \sim \begin{pmatrix}
 1 & 2&3\\
 1 & 4&2 \\
 5 & 2&1 
 \end{pmatrix} \sim \begin{pmatrix}
 1 & 2&3 \\
 1 & 2&4 \\
 1 & 5&2
 \end{pmatrix}.
 \end{equation*}

Since $8<\binom{5}{3},$ and by symmetry between $1,2$ and $3$, as well as between $4$ and $5$, we may assume that 
some list in $\{3,4,5\}, \{1,2,3\}$ or $\{2,3,4\}$ does not appear as a list in $V$. For each of these $3$ lists, a possible choice for $M$ has been presented for which Hall's condition is satisfied for every other list. 
So again we conclude that a list-packing will exist.

 For the following two matrices
 \begin{equation*}
 A_7:= \begin{pmatrix}
 1 & 2&3 \\
 4&5&6 \\
 1 &7&8
 \end{pmatrix}, A_8:=\begin{pmatrix}
 1 & 2&3 \\
 4&5&6 \\
 9 &7&8
 \end{pmatrix}
 \end{equation*}
 Hall's condition can only be not satisfied if the three elements of a list $L(v_j)$ appear in the same column. 
 For each fixed list $L(v_j)$, there is at most one column in the matrix that can contain exactly the colours of $L(v_j)$, due to the fixed row $\vec{c}(u_1)$. Therefore, once these $5$ elements (in that row and column) are fixed, there are only $2$ choices for each of the remaining rows, namely $\vec{c}(u_2)$ and $\vec{c}(u_3)$. Thus, there are at most $4$ matrices that are forbidden by $L(v_j)$.
Given the $8$ lists $(L(v_j))_{j\in[8]}$, the total number of forbidden matrices is at most $8 \cdot 4=32 < 36$. Therefore, there exists choices for $M$ such that the partial packing can be extended.

 We present the matrix $A_9$ in four forms.
 \begin{equation*}
 A_9:=\begin{pmatrix}
 1 & 2&3 \\
 1&2&4\\
 6&1&5
 \end{pmatrix} \sim \begin{pmatrix}
 1 & 2&3 \\
 4&2&1\\
 1&6&5
 \end{pmatrix} \sim \begin{pmatrix}
 1 & 2&3 \\
 4&2&1\\
 5&6&1
 \end{pmatrix} \sim \begin{pmatrix}
 1 & 2&3 \\
 4&1&2\\
 1&6&5
 \end{pmatrix}
 \end{equation*}
 
 If there does not exist a list-packing, then $B$ contains the lists $\{3,4,1\},\{3,4,5\}$ and $\{3,4,6\}$ (first presentation + permuting last row), $\{3,1,5\}, \{3,1,6\}$ (second form), $\{1,4,5\}, \{1,4,6\}$ (third form) and one of out of each of  $\{\{1,2,6\}, \{2,3,5\}\}$ and $\{\{1,2,5\}, \{2,3,6\}\}.$ 
 This is impossible when $B$ only has $8$ vertices and as such a list-packing does exist.
 For matrix
\begin{equation*}
A_{10}:=\begin{pmatrix}
 1 & 2&3 \\
 1&4&5\\
 6&1&7
 \end{pmatrix}
 \sim\begin{pmatrix}
 1 & 2&3 \\
 5&4&1\\
 6&1&7
 \end{pmatrix}
 \end{equation*}
we observe that if 
$(L(v_j))_{1 \le j \le 8}=\{2,3\} \times \{4,5\} \times \{6,7\},$ the second presentation of $A_{10}$, with a transversal of $1$s, results in a valid packing.

To prevent the existence of a packing, it is necessary that for every triplet $T$ in the set $\{2,3\} \times \{4,5\} \times \{6,7\}$ that is not included in any of the lists in partition class $B$, the three triplets in $\binom{[7] \backslash T}{3}$ that contain the number $1$ are all included in the lists of partition class $B$.
It is easy to see that in the latter case, we have more than $8$ lists, which is a contradiction.
For this, one can construct a bipartite graph with the $8$ lists in $\{2,3\} \times \{4,5\} \times \{6,7\}$ on one side and $12$ lists of the form $\{ X \mid 1 \in X \in \binom{[7] \backslash T}{3} \}$ on the other side and observe that every non-empty subset of vertices on the smaller side has a neighbourhood with strictly larger cardinality.

We still need to consider the following remaining matrices:
\begin{equation*}
 A_{11}:=\begin{pmatrix}
 1 & 2&3 \\
 1&2&4 \\
 5&6&7
 \end{pmatrix}, ~A_{12}:=\begin{pmatrix}
 1 & 2&3 \\
 1&4&5\\
 6&2&7
 \end{pmatrix}
 \end{equation*}
 The crucial part is to observe that by Hall's matching theorem, there always exists a common derangement of $L(v_j)$ except possibly if there are two columns both containing $1$ and $2$, or there is a column containing exactly the three elements of $L(v_j)$.

For the matrix $A_{11}$, if $\{3,4,7\}$ is not among the lists $(L(v_i))_{1\le i \le t}$, due to Hall's matchings theorem,
 we can permute $L(v_i)$ into a common derangement of $(1,2,3), (1,2,4)$ and $(5,6,7)$, i.e., a list-packing does exist. 
 The same is true if $\{3,4,5\}$ or $\{3,4,6\}$ does not occur.
 So if no list-packing does exist, we know that $\{3,4,5\},\{3,4,6\},\{3,4,7\}$ all belong to $(L(v_i))_{1\le i \le t}.$
Analogously, in the case of matrix $A_{12}$, if one of the four lists $3 \times \{4,5\} \times \{6,7\}$ does not belong to the lists on $V$, a packing is always possible.
 As such, the matrices with two columns both containing $1$ and $2$ are already both forbidden.
 Every additional list $L(v_j)$ forbids at most $4$ matrices, and so since $4\cdot 8<36$, not all of the permuted matrices are forbidden.
\end{proof}

Possibly, the following bound would extend the classical result for list colouring to list packing in a natural way.

\begin{question}\label{ques:unbalancedcompbip}
 Is it true that $\chi_\ell^\star(K_{d,t}) \le d$ whenever $t < d^{d-1}?$
\end{question}

If this question can be answered affirmative, it also implies that there are infinitely many counterexamples to~\cite[Conj.~15]{KMMP22}, by combining the result with~\cite[Prop.~22(i)]{KMMP22}.

Finally, we remark that up to knowing the precise values for $T'(3),T(3)$ (Proposition~\ref{prop:TT'3}), we know $p(K_{a,b})$ whenever $\min\{a,b\} \le 3$ for $p \in \{ \chi_\ell, \chi_c, \chi_\ell^\star \}.$
The verification of $\chi_c(K_{3,5})=3$ and $\chi_c(K_{3,6})=4$ has been done by brute force verification. See
\url{https://github.com/StijnCambie/ListPackII}, document DP\char`_ K\char`_\{3,5\}.py.

\begin{proposition}\label{prop:chilK3t_b}
 $$\chi_{\ell}(K_{3,t})=
\begin{cases}
2 \mbox{ if } t=2\\
 3 \mbox{ if } 3 \le t \le 26,\\
 4 \mbox{ if } t \ge 27.
\end{cases}
 \chi_c(K_{3,t})=
\begin{cases}
 3 \mbox{ if } 2 \le t \le 5,\\
 4 \mbox{ if } t \ge 6.
\end{cases} 
 \chi_{\ell}^\star(K_{3,t})=
\begin{cases}
 3 \mbox{ if } 2 \le t \le 8,\\
 4 \mbox{ if } t \ge 9.
\end{cases} $$
\end{proposition}

\section{Conclusion}\label{sec:conc}

By investigating the threshold functions $T(d)$ and $T'(d)$ for which $\chi_c^\star(K_{d,t}) \ge 2d$ resp. $2d-1$, we concluded that every positive integer different from $3$ belongs to $\Im(\chi_c^\star).$
The conclusion also would follow once the following is proven, taking into account that $\chi_c^\star(K_{2,2})=4$ and there is no graph for which $\chi_c^\star(G)=3.$

\begin{problem}
 Prove that if $a \ge b \ge 2$, then $\chi_c^\star(K_{a+1,b})\le \chi_c^\star(K_{a,b}) +1$
 and $\chi_c^\star(K_{a+1,a+1})\le \chi_c^\star(K_{a,a}) +1.$
\end{problem}
Note that the resolution of this problem also would imply that $\chi_c^\star(G) \le \Delta(G)+1$ for complete bipartite graphs different from $K_{2,2}$ and as such implying~\cite[Conj.~3]{CCDK22+} for complete bipartite graphs.

For $n$ even, it is conjectured that $\chi_c^\star(K_{n+1})= \chi_c^\star(K_{n})+2,$ while for $a$ sufficiently large in terms of $b$, we know that $\chi_c^\star(K_{a,b+1})= \chi_c^\star(K_{a,b}) +2.$
We conjecture that there is no vertex that can contribute more than $2$.

\begin{conjecture}\label{conj:extrav_max+2}
 Let $G$ be a graph and $v \in V(G)$ a vertex of $G$.
 Then $\chi_c^\star(G \setminus v) \ge \chi_c^\star(G)-2.$
\end{conjecture}

Equivalently, adding a new vertex to a graph $G$ which is adjacent to all other vertices, increases the correspondence packing number by at most two.

We would like to state the following conjecture, which states that, as is believed for Ramsey numbers (and we believe to be true for $\chi_{\ell}, \chi_c$ and $\chi_\ell^\star$ as well), the maximum values are obtained on the diagonal.

\begin{conjecture}
 When $a>b,$ then $\chi_c^\star(K_{a+1,b})\le \chi_c^\star(K_{a,b+1})$. 
\end{conjecture}

 This conjecture being true, would imply that to study the analogue of the first open question in~\cite{ERT80} it is sufficient to consider the packing numbers of balanced complete bipartite graphs.
 In~\cite{ERT80}, the authors consider the value $N(2,k)$ which is the smallest order of a (complete) bipartite graph $G$ for which $\chi_{\ell}(G)>k.$ Table~\ref{tab:tableN2kvalues} presents the values for the extension $N_\rho(2,k)$, which is the analogue presenting the smallest order of a (complete) bipartite graph $G$ for which $\rho(G)>k$ for the four studied graph parameters $\rho$. Note that already for list packing, we did not determine the precise value for $k=3$.
 
\begin{table}[h]
\begin{center}
 \begin{tabular}{ | l | l | l |l|l| }
 \hline
$k\setminus \rho$& $\chi_{\ell}$ &$\chi_{\ell}^\star$ & $\chi_c$ & $\chi_c^\star$ \\
 \hline 
 2& 6 & 4&4&4\\ 
 3& 14 & 11? &8&4\\
 \hline
 \end{tabular}
\end{center}
\caption{Values for $N_\rho(2,k)$ for $k \in \{2,3\}$ and $\rho \in \{\chi_{\ell},\chi_c, \chi_{\ell}^\star,\chi_c^\star\}.$}\label{tab:tableN2kvalues}
\end{table}

The proof for $N(2,3)=14$ can be found in~\cite{HMT96}.
A brute force computer verification has shown that $\chi_\ell^\star(K_{5,6})\ge 4$, while for every list-assignment with lists in $\binom{[5]}{3}$, the graph $K_{5,5}$ is list-packable. See
\url{https://github.com/StijnCambie/ListPackII}, document chi$\char`_\ell\char`^ (K\char`_6,5)$.py.
The (up to renaming colours) unique lists for which $K_{6,5}$ is not $3$-list-packable are
$(L(v_i))_{i \in [6]}=\{ A \cup \{5\} \mid A \in \binom{[4]}{2}\}$ and $(L(u_i))_{i \in [5]}=\binom{[4]}{3}\cup \{1,2,5\}.$
It is easy to see that no proper colouring exists assigning the colour $1$ to $u_5$ ($L(u_5)=\{1,2,5\}$) and thus there is no $L$-packing of $K_{6,5}.$
For $N_{\chi_c}(2,3)=8$, we refer to Proposition~\ref{prop:chilK3t_b} and a computer verification that $\chi_c(K_{4,4})=3.$ See
\url{https://github.com/StijnCambie/ListPackII}, document chi\char`_c(K\char`_4,4).py.

 We remind the reader that also studying the unbalanced case, as initiated in~\cite{ACK20+,CaKa20+} might be interesting and in particular, Question~\ref{ques:unbalancedcompbip} might be an elegant extension of a classical boundary case.

The probability of existence of an $L$-colouring has also been studied for bipartite graphs with random lists~\cite{KN06}, providing sharp thresholds. It would be interesting to see if the threshold for the packing variant is very different,
as an indication of the substantial difference between finding one proper colouring and a packing.

\section*{Acknowledgement}
The authors would like to express their gratitude to the organisers of the $17^{th}$ (virtual) Graduate Student Combinatorics Conference, which lead to this collaboration.

\paragraph{Open access statement.} For the purpose of open access,
a CC BY public copyright license is applied
to any Author Accepted Manuscript (AAM)
arising from this submission.

\bibliographystyle{abbrv}
\bibliography{listpack}

\end{document}